\documentclass[a4paper,12pt]{amsart}
\usepackage{amssymb,graphicx}
\DeclareSymbolFont{script}{U}{eus}{m}{n}
\DeclareMathSymbol{\Wedge}{0}{script}{"5E}
\DeclareMathAlphabet{\mathrmsl}{OT1}{cmr}{m}{sl}
\newcommand{\R}{\mathbb R}

\newcommand{\trace}{\mathop{\mathrmsl{tr}}\nolimits}
\newcommand{\grad}{\mathop{\mathrmsl{grad}}\nolimits}
\newcommand{\Id}{\mathrmsl{Id}}
\newcommand{\Scal}{\mathrmsl{Scal}}
\newcommand{\Sol}{\mathrmsl{Sol}}
\newcommand{\Lam}{{\mathrmsl\Lambda}}
\newcommand{\Hom}{\mathrmsl{Hom}}
\newcommand{\Nu}[1]{{\mathcal N}_{#1}}
\newcommand{\nablat}{\widetilde\nabla}
\newcommand{\gt}{\widetilde g}
\newcommand{\C}{{\mathbb C}}
\renewcommand{\d}{{\mathrmsl d}}
\newcommand{\CC}{\mathrmsl{C}_\C}
\newcommand{\cC}{{\mathcal C}}
\newcommand{\cD}{{\mathcal D}}
\newcommand{\cL}{{\mathcal L}}
\theoremstyle{plain}
\newtheorem{thm}{Theorem}
\newtheorem*{thm*}{Theorem}
\newtheorem{lem}{Lemma}
\newtheorem{prop}{Proposition}
\newtheorem{cor}{Corollary}

\theoremstyle{definition}
\newtheorem{defn}{Definition}

\theoremstyle{remark}
\newtheorem{rem}{Remark}

\title[Curvature and the c-projective mobility of K\"ahler metrics]
{Curvature and the c-projective mobility\\ of K\"ahler metrics
with hamiltonian $2$-forms}
\author{David M.J. Calderbank}
\address{DMJC: Department of Mathematical Sciences\\ University of Bath\\
Bath BA2 7AY\\ UK.}
\email{D.M.J.Calderbank@bath.ac.uk}
\author{Vladimir S. Matveev}
\author{Stefan Rosemann}
\address{VSM, SR: Institute of Mathematics\\ FSU Jena\\ 07737 Jena Germany}
\email{vladimir.matveev@uni-jena.de, stefan.rosemann@uni-jena.de}
\begin{document}

\begin{abstract}
The mobility of a K\"ahler metric is the dimension of the space of metrics
with which it is c-projectively equivalent. The mobility is at least two if
and only if the K\"ahler metric admits a nontrivial hamiltonian $2$-form.
After summarizing this relationship, we present necessary conditions for a
K\"ahler metric to have mobility at least three: its curvature must have
nontrivial nullity at every point. Using the local classification of K\"ahler
metrics with hamiltonian $2$-forms, we describe explicitly the K\"ahler
metrics with mobility at least three and hence show that the nullity condition
on the curvature is also sufficient, up to some degenerate exceptions. In an
Appendix, we explain how the classification may be related, generically, to
the holonomy of a complex cone metric.
\end{abstract}

\maketitle
%\tableofcontents

\section*{Introduction}

This article weaves together two threads in K\"ahler geometry which have been
running in parallel for 40--60 years with remarkably little interaction, given
their common themes.

The first thread concerns a notion of projective equivalence between K\"ahler
metrics. The classical notion is too strong when applied to K\"ahler metrics: 
if two metrics that are hermitian with respect to the same almost complex
structure have the same geodesics, they have the same Levi-Civita
connection.  In 1954, Otsuki and Tashiro~\cite{OtTash54} introduced a complex,
but non-holomorphic, version of projective equivalence, which acquired the
unfortunate name of ``holomorphically projective'' or ``h-projective''
equivalence in the literature. We prefer the term ``c-projective'', which is
intended to suggest ``complex projective'', without implying that the
geometry is holomorphic.

\begin{defn}\label{defn:cproj} Let $(M,J)$ be a complex manifold of real
  dimension $2m\geq 4$.  Then two $J$-hermitian K\"ahler metrics $g,\gt$ on
  $M$, with Levi-Civita connections $\nabla, \nablat$, are called
  \emph{c-projectively equivalent} if there is a $1$-form $\Phi$ such that
\begin{equation}\label{eq:cproj}
\nablat_X Y-\nabla_X Y=\Phi(X)Y+\Phi(Y)X-\Phi(JX)JY-\Phi(JY)JX
\end{equation}
for all vector fields $X,Y$.
\end{defn}
This notion has been extensively studied by Russian and Japanese schools
(see~\cite{Mikes} for a list of references up to 1998). One common theme
has been the relationship between special curvature properties of a K\"ahler
metric and the existence of metrics c-projectively equivalent to it
(e.g.~\cite{IshTach61}).

\smallbreak
The second thread concerns the explicit construction of ``optimal'' K\"ahler
metrics on complex manifolds, generalizing the constant curvature metrics used
in the uniformization of Riemann surfaces. The idea to seek such metrics goes
back to Calabi's famous conjectures in the 1950s (e.g.,~\cite{Calabi54}), but
the problem was attacked primarily using analytical methods until the late
1970s. Then Calabi provided fresh impetus by introducing the notion of an
extremal K\"ahler metric and constructing explicit examples on total spaces of
complex projective line bundles~\cite{Calabi79,Calabi82}.  Calabi's
construction has been refined and extended considerably by many authors (see
e.g.,~\cite{Abreu,HwSing}), providing a rich supply of K\"ahler metrics with
special curvature properties (such as extremal K\"ahler metrics). These
generalizations have in common that they introduce first order structure to
simplify the second (and higher) order partial differential equations that
describe curvature. A single source for this structure was identified
in~\cite{ACG1}, where it was observed that Calabi's construction and its
generalizations reflect the presence of a nontrivial solution to an
overdetermined linear differential equation, called a hamiltonian $2$-form.

\begin{defn} Let $(M, g, J,\omega)$ be a K\"ahler manifold of real dimension
$2m\geq 4$. Then a (real) $J$-invariant $2$-form $\phi$ on $M$ is
{\it hamiltonian} if
\begin{equation} \label{eq:ham}
\nabla_X \phi = \frac12 (\d \trace_\omega\phi \wedge JX^\flat
- J\d \trace_\omega \phi \wedge X^\flat)
\end{equation}
for all vector fields $X$, where $X^\flat=g(X,\cdot)$, $JX^\flat=-X^\flat\circ
J=(JX)^\flat$, and $\trace_\omega\phi=g(\omega,\phi)$ is the trace of
$\phi$ with respect to the K\"ahler form $\omega$.
\end{defn}
K\"ahler manifolds with hamiltonian $2$-forms are classified locally
in~\cite{ACG1} and globally in~\cite{ACGT2}, with applications to extremal
K\"ahler metrics in~\cite{ACGT3}.

\smallbreak

The origins of the present article are somewhat serendipitous. In April 2011,
the first author was asked to referee the article~\cite{MatRos1} by the second
and third authors, which proves that the only compact c-projective manifold
with a one parameter subgroup of ``essential'' symmetries is complex
projective space. This drew the first author's attention to the ``main
equation'' of c-projective equivalence (equation~\eqref{eq:main} below), which
is manifestly equivalent to the equation for hamiltonian $2$-forms (see
Remark~\ref{rem:affine}).

As noted in the published version of~\cite{MatRos1}, this equivalence has two
main ramifications. First, the organizing principle observed in~\cite{ACG1} to
underpin explicit constructions of K\"ahler metrics coincides with the notion
of a c-projectively equivalent metric, a topic studied independently for many
years previously. Secondly, the classification results in~\cite{ACG1,ACGT2}
solve open problems in the theory of c-projective equivalence, as well as
providing new examples.

Our interest here is in a third ramification: although the methodologies
employed in the theories of c-projective equivalence and hamiltonian $2$-forms
have a large overlap (e.g., as both depend upon the theory of overdetermined
PDEs of finite type), they have quite different flavours which might be
combined with profit to prove new results. This article is a first attempt to
exploit both theories in this way.

We focus on the \emph{mobility} $D(g,J)$ of a K\"ahler metric $g$ on $(M,J)$,
which is the dimension of the space $\Sol(g,J)$ of solutions of
equation~\eqref{eq:main}---or equivalently equation~\eqref{eq:ham}. Since the
identity map $\Id$ (corresponding to the K\"ahler form $\omega$) is always a
solution, $D(g,J)\geq 1$, and the presence of an independent solution (or a
nontrivial hamiltonian $2$-form) means equivalently that $D(g,J)\geq 2$.

Our plan is to study the case $D(g,J)\geq 3$, using \cite[Theorem 5]{FKMR},
quoted as Theorem~\ref{thm:extsystem} below, which states that any such
K\"ahler metric $g$ is $\CC(B)$ (for some $B\in\R$) in the sense of
Definition~\ref{defn:CCB} (unless all solutions of~\eqref{eq:main} are
parallel).  The converse is not true: it is straightforward to construct
$\CC(B)$ metrics with mobility $2$ (e.g., using the cone construction
described in the appendix---see~\S\ref{subsec:Dgeq3_rederived}).  In
Theorems~\ref{thm:equivalence} and~\ref{thm:condforCCB} we establish necessary
and sufficient conditions for a K\"ahler metric to be $\CC(B)$, and then, in
Theorem \ref{thm:Dgeq3}, describe the additional conditions such that a
$\CC(B)$ metric $g$ has mobility $D(g,J)\geq 3$.

Whereas Theorem~\ref{thm:equivalence} draws upon curvature conditions from the
theory of c-projective equivalence, Theorem~\ref{thm:condforCCB} uses
hamiltonian $2$-form methods. It follows, in Corollary~\ref{cor:degree}, that
an extremal K\"ahler metric with mobility $\geq 3$ must have constant scalar
curvature.

Our results are closely related to the cone construction of~\cite{MatRos2},
cf.~\cite{Arm,Mikes}, discussed in Appendix A. More precisely, for $\CC(B)$
metrics with $B<0$ (and we may assume $B=-1$ by rescaling), this construction
gives an explicit isomorphism between $\Sol(g,J)$ and the space of parallel
hermitian endomorphisms on a complex cone $(\hat M,\hat g, \hat J)$ over
$(M,g,J)$, which we summarize in \S\ref{subsec:cone}.  The cone is a K\"ahler
manifold of dimension $\dim_{\C}M+1$ and $(M,g,J)$ may be recovered from it by
taking a K\"ahler quotient. It is known, at least since Eisenhart \cite{ei},
that the existence of a parallel hermitian endomorphism $\hat A$ on $\hat M$
is (locally) equivalent to a decomposition of $\hat M$ into a direct product
of K\"ahler manifolds.

In \S\ref{subsec:quotient}, we derive a formula for the K\"ahler quotient
metric $g$ in terms of radial and angular coordinates on $\hat M$ coming from
the decomposition of $\hat M$ induced by $\hat A$.  In
\S\ref{subsec:locclass_rederived}, we (partially) rederive the local
classification formula \eqref{eq:locclass} for $g$ relative to $A\in
\Sol(g,J)$ corresponding to $\hat A$; this yields another proof of (one
direction of) Theorem \ref{thm:condforCCB} by a direct calculation---see
Proposition~\ref{prop:nec}.  In~\S\ref{subsec:Dgeq3_rederived} we use the cone
construction to give an alternative proof of Theorem~\ref{thm:Dgeq3} for a
$\CC(-1)$ metric.

\section{C-projective equivalence and hamiltonian $2$-forms}
\label{sec:basics}

\subsection{C-projective equivalence and $\CC(B)$ metrics}
\label{subsec:basicequations}

Let $(M,J)$ be a complex manifold of real dimension $2m\geq 4$. For
$J$-hermitian metrics $g,\gt$ on $M$, we introduce the nondegenerate
$(g,J)$-hermitian (i.e., $g$-symmetric, $J$-complex-linear) endomorphism
\begin{equation}\label{eq:defA}
A(g,\gt):=\left(\frac{\det\gt}{\det g}\right)^{\frac{1}{2(m+1)}} \gt^{-1}g,
\end{equation}
where we view $g,\gt\colon TM\rightarrow T^* M$ as bundle isomorphisms.  A
fundamental observation by Domashev and Mike\v s~\cite{DomMik1978} is that $g$
and $\gt$ are c-projectively equivalent if and only if there is a vector field
$\Lam$ such that $A=A(g,\gt)$ satisfies the ``main equation''
\begin{align}\label{eq:main}
\nabla_X A=X^\flat\otimes\Lam+\Lam^\flat\otimes X
+JX^\flat\otimes J\Lam+J\Lam^\flat\otimes JX.
\end{align}
Conversely, a nondegenerate solution $A$ of \eqref{eq:main} determines a
K\"ahler metric
\begin{equation}
\gt=(\det A)^{-\frac{1}{2}}gA^{-1}
\end{equation}
(obtained by solving \eqref{eq:defA} with respect to $\gt$) c-projectively
equivalent to $g$. Since $\Id$ is always a solution of~\eqref{eq:main}, we can
add a multiple of $\Id$ to any solution $A$ to obtain (at least locally) a
solution which is nondegenerate. In this sense, the solutions $A$ of
\eqref{eq:main} are (locally, generically) in bijection with K\"ahler metrics
$\gt$ that are c-projectively equivalent to $g$.

\begin{defn}
The space of hermitian endomorphisms $A$ satisfying \eqref{eq:main} will be
denoted by $\Sol(g,J)$. The \emph{mobility}\footnote{In the classical
c-projective literature, this is known as the ``degree of mobility''.}
$D(g,J)$ of $(M,g,J)$ is the dimension of $\Sol(g,J)$.
\end{defn}

\begin{rem}\label{rem:affine}
Obviously, two metrics $g,\gt$ are affinely equivalent ($\nablat=\nabla$) if
and only if the endomorphism $A=A(g,\gt)$ is parallel.  By \eqref{eq:main}, if
the metrics are c-projectively equivalent, they are affinely equivalent if and
only if the vector field $\Lam$ is identically zero.

Taking the trace on both sides of equation \eqref{eq:main}, shows that 
\begin{align}\label{eq:lambda}
\Lam=\frac{1}{4}\grad_g\trace A,
\end{align}
hence~\eqref{eq:main} is a linear PDE system on $A$, which is equivalent to
equation~\eqref{eq:ham} for a hamiltonian $2$-form $\phi$ by writing
$g(AX,Y)=\phi(X,JY)$.

In \cite{ACG1,MatRos1}, the nonconstant eigenvalues $\xi_1,\ldots \xi_\ell$ of
$A$, considered as functions on $M$, are shown to be continuous, and smooth on
a dense open subset $M^{0}$. Moreover, their (complex) multiplicity on this
subset is one. Thus we can express $\Lam$ on $M^{0}$ as
\begin{align}\label{eq:lambda1}
\Lam=\frac{1}{2}\sum_{i=1}^{\ell}\grad_g \xi_i.
\end{align}
For each nonconstant eigenvalue $\xi_i$ of $A$, $\grad_g\xi_i$ lies in the
corresponding eigenspace (see \cite{ACG1,MatRos1}). Hence the vanishing of
$\Lam$ is equivalent to all eigenvalues of the endomorphism $A$ (considered as
functions on the manifold) being constant.
\end{rem}
An important standard result in c-projective geometry is the fact that
$J\Lam$ is Killing.
\begin{lem}\label{lem:lambdakilling}
Let $(M,g,J)$ be a K\"ahler manifold of real dimension $2m\geq 4$. Then for
any $A\in\Sol(g,J)$, the corresponding vector field $\Lam$ is holomorphic, and
$J\Lam$ is a Killing vector field---equivalently $\nabla\Lam$ is
$(g,J)$-hermitian.
\end{lem}
\begin{proof} This is well known: see \cite[Eq. (13)]{DomMik1978},
\cite[Proposition 3]{ACG1} and \cite[Corollary 3]{FKMR}.
\end{proof}
As the introduction explains, our study builds on the following theorem.
\begin{thm}\cite{FKMR}\label{thm:extsystem}
Let $(M,g,J)$ be a connected K\"ahler manifold of real dimension $2m\geq 4$
and mobility $D(g,J)\geq 3$. Then there is a unique $B\in\R$ such that for
every $A\in\Sol(g,J)$, with corresponding vector field $\Lam$, there is a
function $\mu$ such that the system
\begin{equation}\label{eq:extsystem}\begin{split}
\nabla_X A
&=X^\flat \otimes \Lam+\Lam^\flat \otimes X
+JX^\flat\otimes J\Lam+J\Lam^\flat\otimes JX,\\
\nabla\Lam&=\mu\Id+BA,\\
\nabla\mu&=2B \Lam^\flat
\end{split}
\end{equation}
holds at every point of $M$.
\end{thm}
\begin{rem} If for $A\in \Sol(g,J)$, $A\neq \mathrm{const}\cdot \Id$, with
corresponding vector field $\Lam$, there exists a function $\mu$ such that
$(A,\Lam,\mu)$ solves \eqref{eq:extsystem} for a certain constant $B$, then
this holds for any other element $\tilde A\in \Sol(g,J)$. This is clear if
$\tilde A$ is a linear combination of $\Id$ and $A$ and follows from Theorem
\ref{thm:extsystem} if $\Id,A,\tilde A$ are linearly independent.
\end{rem}
\begin{defn}\label{defn:CCB} Let $B$ be a real number. A K\"ahler metric
$(g,J)$ is called\footnote{Here ``$\CC$''
    suggests constant/curvature/cone and complex/c-projective, and replaces
    the term ``$K_n(B)$'', often used in the classical c-projective
    literature, in which $K_n$ denotes a K\"ahler $n$-manifold.
} $\CC(B)$ if it admits a solution $(A,\Lam,\mu)$ to the
system~\eqref{eq:extsystem} with $\Lam$ not identically zero.
\end{defn}
\begin{rem}
  In Definition~\ref{defn:CCB} we require $B$ to be a constant. If $B$ is
  initially assumed to be a function, it turns out that this function must be
  (locally) constant provided there exists at almost every point a nonzero
  vector contained in the $B$-nullity of the curvature, see Definition
  \ref{defn:nullity} and Theorem \ref{thm:equivalence} below.
\end{rem}
\begin{rem} Neither equation~\eqref{eq:ham} nor equation~\eqref{eq:main}
provide the most natural formulation of c-projective equivalence and mobility
because they treat the metrics $g$ and $\gt$ asymmetrically. This can be
remedied by observing that the defining equation~\eqref{eq:cproj} for
c-projective equivalence is really an equivalence relation between complex
affine connections (connections $\nabla$ on $TM$ with $\nabla J=0$).  A
\emph{c-projective structure} on a complex manifold $(M,J)$ is a c-projective
equivalence class of such complex affine connections. Equation~\eqref{eq:main}
can be rewritten without reference to a background metric $g$ replacing $A$
with the metric $h$ on $T^*M$ defined by $h(\alpha,\beta) =g(\alpha\circ
A,\beta)$. Then equation~\eqref{eq:main} becomes
\begin{equation*}
\nabla_X h=X \otimes\Lam+\Lam\otimes X+JX\otimes J\Lam+J\Lam\otimes JX
\end{equation*}
(for all vector fields $X$) and this equation for $h$ depends only on the
c-projective class of $\nabla$ provided that $h$ is viewed as a section of
$\cL^*\otimes S^2TM$, where $\cL^{\otimes(m+1)}=\Wedge^{2m}TM$.

This viewpoint is developed in detail in the forthcoming survey~\cite{CEMN} on
c-projective geometry---see also~\cite{Yoshi78}. For the present article, we
shall always have in mind a background metric, and so we do not pursue this
reformulation any further.
\end{rem}

\subsection{The classification of hamiltonian $2$-forms}
\label{subsec:localclass}

According to \cite{ACG1}, a K\"ahler metric $(g,J,\omega)$ admitting a
hamiltonian $2$-form---or equivalently an $A\in\Sol(g,J)$---is locally a
bundle over a product of K\"ahler $2m_\eta$-manifolds indexed by the constant
eigenvalues $\eta$ of $A$ ($m_\eta$ being the multiplicity of $\eta$), whose
``orthotoric'' fibres are totally geodesic with the nonconstant eigenvalues
$\xi_1,\ldots\xi_\ell$ of $A$ as coordinates. On a dense open set, we may write
\begin{align}\label{eq:locclass}
g&=\underbrace{\sum_{\eta} p_{\mathrm{nc}}(\eta)g_{\eta}}_{\mbox{\small base metric}}
+\underbrace{\sum_{i=1}^\ell\frac{\Delta_j}{\Theta_j(\xi_j)}\d\xi_j^2
+\sum_{j=1}^\ell \frac{\Theta_j(\xi_j)}{\Delta_j}\Bigl(\sum_{r=1}^\ell
\sigma_{r-1}(\hat\xi_j)\theta_r\Bigr)^2}_{\mbox{\small fibre metric}},\\
\omega&=\sum_\eta p_{\mathrm{nc}}(\eta) \omega_{\eta}
+\sum_{r=1}^\ell \d\sigma_r\wedge \theta_r,\qquad\text{with}\qquad
\d\theta_r=\sum_\eta(-1)^r\eta^{\ell-r}\omega_{\eta},
\end{align}
where $p_{\mathrm{nc}}(t)=\prod_{i=1}^\ell(t-\xi_i)$, $\sigma_r$ is the $r$th
elementary symmetric function of $\{\xi_1,\ldots\xi_\ell\}$, $\sigma_{r-1}
(\hat{\xi} _j)$ is the $(r-1)$st such function of $\{\xi_k:k\neq j\}$,
$\Delta_j=\prod_{k\neq j}(\xi_j-\xi_k)$, and
\begin{equation}
J \d \xi _j = \frac{\Theta_j (\xi_j)}{\Delta_j} \,
\sum_{r=1}^\ell \sigma_{r-1} (\hat{\xi}_j) \,\theta_r,\qquad
J\theta_r = (-1)^r \,\sum_{j=1}^\ell \frac{\Delta_j}{\Theta_j(\xi_j)}
\xi_j^{\ell-r}\,\d\xi_j.
\end{equation}
For any metric of this form,
\begin{align*}
\phi:&= \sum_\eta \eta\, p_{\mathrm{nc}}(\eta) \omega_\eta
+ \sum_{j=1}^\ell \xi_j\,
\d\xi_j\wedge\Bigl(\sum_{r=1}^\ell\sigma_{r-1}(\hat\xi_j)\theta_r\Bigr)\\
&= \sum_\eta \eta\, p_{\mathrm{nc}}(\eta) \omega_\eta
+\sum_{r=1}^\ell (\sigma_r \d\sigma_1 -\d\sigma_{r+1})\wedge\theta_r
\end{align*}
is a hamiltonian $2$-form. The extension of this local classification to 
pseudo-riemannian metrics is subject of the forthcoming paper~\cite{BMR}.

Curvature properties of the metric $g$ in~\eqref{eq:locclass} are also
computed in~\cite{ACG1}, to which we refer for details and explanations.  Let
$p_c(t)=\prod_\eta (t-\eta)^{m_\eta}$ be the (monic) polynomial whose roots
are the constant eigenvalues $\eta$ of $\phi$, counted with multiplicity.
\begin{enumerate}
\item $g$ is Bochner-flat if and only if the functions $\Theta_j(t)$ are
  equal, given by a polynomial $\Theta(t)$ of degree $\leq \ell+2$, with
  $\Theta(\eta)=0$ for all constant eigenvalues $\eta$, and the base metrics
  $g_\eta$ have constant holomorphic sectional curvature (CHSC), given by
  $-\Theta'(\eta)$. The metric $g$ is itself CHSC if and only if in addition
  $\deg\Theta(t)\leq \ell+1$.
\item $g$ is weakly Bochner-flat if and only if the functions
  $(p_c\Theta_j)'(t)/p_c(t)$ are equal, given by a polynomial $\Psi(t)$ of
  degree $\leq\ell+1$, and the base metrics $g_\eta$ are K\"ahler--Einstein,
  with $\frac1{m_\eta}\Scal_{g_\eta}=-\Psi(\eta)$. The metric $g$ is itself
  K\"ahler--Einstein if and only if in addition $\deg\Psi(t)\leq \ell$.
\end{enumerate}

In particular (applying (1) fibrewise, using the case that there are no
constant eigenvalues), the orthotoric fibres have CHSC if and only if the
functions $\Theta_j(t)$ are equal to a common polynomial of degree $\leq
\ell+1$.

It will also be useful to recall from~\cite{ACG1} that there is a
``Gray--O'Neill'' formula~\cite{Gray,ONeill} for the Levi-Civita connection
of $g$ in terms of the fibre and base metrics, where the Gray--O'Neill tensor
of the horizontal distribution is given by
\begin{equation}\label{eq:GrayONeill}
2C(X,Y)= \sum_{r=1}^\ell \bigl(\Omega_r(X,Y)J\Lam_r - \Omega_r(JX,Y)\Lam_r\bigr)
\end{equation}
for $\Omega_r=\sum_\eta(-1)^r\eta^{\ell-r}\omega_{\eta}$ and
$\Lam_r=\grad_g\sigma_r$.

\section{Curvature nullity and the extended system}
\label{sec:curvaturenullity}

Let $R\in \Omega^{2}(M,\mathfrak{gl}(TM))$ denote the curvature of the
K\"ahler manifold $(M,g,J)$,
\begin{equation*}
R(X,Y)Z=(\nabla_X \nabla_Y - \nabla_Y \nabla_X-\nabla_{[X,Y]})Z,
\end{equation*}
and let 
\begin{align}\label{eq:constholom}
K(X,Y)=\tfrac{1}{4}(Y^\flat \otimes X-X^\flat\otimes Y
+JY^\flat\otimes JX-JX^\flat\otimes JY+2g(X,JY)J)
\end{align}
be the algebraic curvature tensor of constant holomorphic sectional curvature.

\begin{lem}\label{lem:intcond}
Let $(M,g,J)$ be a K\"ahler manifold of real dimension $2m\geq 4$. Then every
$A\in\Sol(g,J)$ satisfies the identity
\begin{align}\label{eq:intcond}
[R(X,Y),A]=-4 [K(X,Y),\nabla\Lam]
\end{align}
at every point for all tangent vectors $X,Y$.
\end{lem}
\begin{proof}
Equation \eqref{eq:intcond} is well known in the theory of c-projectively
equivalent metrics, see for example \cite{DomMik1978,Mikes}. To prove it,
consider the identity
\begin{align}\label{eq:intcondendo}
[R(X,Y),A]=\nabla_X(\nabla A)_Y-\nabla_Y(\nabla A)_X
\end{align}
which holds for any endomorphism $A\in\Gamma(\mathfrak{gl}(TM))$. Assuming
that $A\in \Sol(g,J)$, we can replace the covariant derivatives of $A$
in \eqref{eq:intcondendo} with \eqref{eq:main}, to derive an integrability
condition for \eqref{eq:main}. A straightforward calculation yields the
desired equation \eqref{eq:intcond}. We note that we have to use that
$\nabla\Lam$ commutes with $J$, see Lemma \ref{lem:lambdakilling}.
\end{proof}

\begin{defn}\label{defn:nullity}
For $p\in M$ and $B\in\R$, the $B$-nullity space of the curvature $R$ at $p$
is the linear space
\begin{equation}
N(B)_p=\{Z\in T_p M:\Nu{B}(X,Y)Z=0\,\,\forall X,Y\in T_p M\},
\end{equation}
where $\Nu{B}(X,Y)=R(X,Y)+4BK(X,Y)$.
\end{defn}
\begin{rem}\label{rem:nullity}
Since $g(\Nu{B}(\cdot,\cdot)\cdot,\cdot)$ is a section of $S^2(\Wedge^2T^*M)$,
$N(B)_p$ is the set of $Z\in T_pM$ whose contraction into any entry of
$g(\Nu{B}(\cdot,\cdot)\cdot,\cdot)$ is zero. Note also that $N(B)_p$ is
$J$-invariant, i.e., a complex linear subspace of $T_p M$.
\end{rem}
\begin{rem}
The real number $B$ in the definition of the nullity is unique: 
if $Z\in N(B)_p$ and $Z'\in N(B')_p$ are nonzero vectors, then $B=B'$.
To see this, we replace $X$ by $Z'$ in the nullity condition for $Z$, and
apply the nullity condition for $Z'$ to obtain $(B-B')K(Z,Z')=0$. Hence,
$B=B'$ or $K(Z,Z')=0$. The last equation implies $Z'$ is a multiple of $Z$.
Thus $(B-B')K(X,Y)Z=0$ for all vectors $X,Y$, which, for $Z$ nonzero, forces
$B=B'$.

However, $B$ may depend on the point $p$, and (of course) the metric $g$.
\end{rem}

\begin{prop}\label{prop:key-identities} Let $(M,g,J)$ be a K\"ahler manifold
of real dimension $2m\geq 4$, and let $A\in\Sol(g,J)$ with corresponding
vector field $\Lam$. Then for any functions $B,\mu$, we have
\begin{equation}\label{eq:second}
[K(X,Y),\nabla\Lam-BA-\mu\Id]+\tfrac1 4 [\Nu{B}(X,Y),A]=0
\end{equation}
and, if $B$ and $\mu$ are smooth,
\begin{equation}\label{eq:third}
\nabla_X(\nabla\Lam-BA-\mu\Id)+J\Nu{B}(X,J\Lam)
+(\nabla_X\mu-2Bg(\Lam,X))\Id+\d B(X)A=0.
\end{equation}
\end{prop}
\begin{proof} Equation~\eqref{eq:second} is immediate from
Lemma~\ref{lem:intcond}\eqref{eq:intcond}. Recall from
Lemma~\ref{lem:lambdakilling} that $J\Lam$ is a Killing vector field, and
hence $\nabla_X\nabla\Lam=-J\nabla_X\nabla J\Lam=-JR(X,J\Lam)$ (by
the standard formula $\nabla_X\nabla K=R(X,K)$, $X\in TM$, which holds for any
Killing vector field $K$, see \cite{Kostant}). Equation~\eqref{eq:third}
follows from this by expanding $\nabla_X(\nabla\Lam-BA-\mu\Id)$ and
substituting for $\nabla_X A$ from equation~\eqref{eq:main}.
\end{proof}
\begin{lem}\label{lem:Kbracket} Let $Q$ be a hermitian endomorphism and $Z$
a nonzero tangent vector at $p\in M$ such that $[K(X,Z),Q]=0$ for all $X\in
T_p M$. Then $Q$ is a multiple of the identity.
\end{lem}
\begin{proof} We may assume $Q$ is tracefree and prove it vanishes. By
definition~\eqref{eq:constholom} of $K$,
\begin{equation}\label{eq:bracket}
[Z^\flat\otimes X-X^\flat \otimes Z+JZ^\flat\otimes JX-JX^\flat\otimes JZ,Q]=0.
\end{equation}
Let $e_1,\ldots e_{2m}$ be an orthonormal frame of $T_p M$. We take a trace
by applying~\eqref{eq:bracket} to $e_i$ with $X=e_i$ and summing over $i$.
Since $Q$ and $Q\circ J=J\circ Q$ are trace-free, and $Q$ is hermitian, we
obtain (with summation understood)
\begin{equation*}
0=g(Z,Qe_i)e_i-g(Z,e_i)Qe_i+g(e_i,e_i)QZ+g(JZ,Qe_i)Je_i-g(JZ,e_i)QJe_i =2m QZ.
\end{equation*}
Thus $QZ=0$, which we substitute into~\eqref{eq:bracket} to obtain:
\begin{equation*}
Z^\flat\otimes QX + (QX)^\flat\otimes Z + JZ^\flat\otimes Q(JX) +
Q(JX)^\flat\otimes JZ=0.
\end{equation*}
For any $Y\in\mathrm{span}\{Z,JZ\}^{\perp}$ this yields (using that $Q$ is
hermitian)
\begin{equation*}
0=g(QX,Y)Z+g(Q(JX),Y)JZ=g(X,QY)Z+g(JX,QY)JZ.
\end{equation*}
Since $Z\neq 0$, $Q$ vanishes on $\mathrm{span}\{Z,JZ\}^{\perp}$. But $Q$
vanishes on $\mathrm{span}\{Z,JZ\}$, so $Q=0$.
\end{proof}
\begin{thm}\label{thm:equivalence}
Let $(M,g,J)$ be a connected K\"ahler manifold of real dimension $2m\geq 4$.
Then for any $A\in\Sol(g,J)$ with corresponding vector field $\Lam$ such that
$A$ is not parallel \textup(equivalently, $\Lam\neq 0$\textup) the following
statements are equivalent\textup:
\begin{enumerate}
\item there is a constant $B$ such that $[\Nu{B}(X,Y),A]=0$ for all vector
  fields $X,Y$\textup;
\item there is a constant $B$ and a smooth function $\mu$ such that
  $\nabla\Lam=BA+\mu\Id$\textup;
\item there is a constant $B$ and a smooth function $\mu$ such that $A$
  satisfies the extended system~\eqref{eq:extsystem}\textup;
\item there is a constant $B$ such that $\Lam$ is in the $B$-nullity space
  $N(B)_p$ at every $p\in M$---equivalently $\Nu{B}(X,J\Lam)=0$ for all $X\in
  TM$\textup;
\item at every point $p$ of a dense subset, there is a real number $B=B(p)$
  such that the $B$-nullity space $N(B)_p$ is nonzero\textup;
\item there is a constant $B$ such that for any open subset $U$ of $M$ and any
  eigenvalue $\xi$ of $A$ smoothly defined on $U$, $\grad_g\xi$ is in the
  $B$-nullity of the curvature on $U$.
\end{enumerate}

If for given $B$, these conditions hold for some non-parallel $A\in\Sol(g,J)$,
then they hold for all $A\in\Sol(g,J)$ \textup(with the same constant
$B$\textup). In particular, the metric $g$ is $\CC(B)$.
\end{thm}
\begin{proof} (1)$\Leftrightarrow$(2) by equation~\eqref{eq:second}: if
$\nabla\Lam-BA$ commutes with $K(X,Y)$ for all $X,Y\in T_p M$, then it
commutes with all skew-hermitian endomorphisms of $T_p M$ and is hence a
multiple of the identity at $p$.

(2)$\Leftrightarrow$(3) by equation~\eqref{eq:third}, which reduces to
\begin{equation*}
g(\Nu{B}(X,J\Lam)Y,Z)=(\nabla_X\mu-2Bg(\Lam,X))g(JY,Z)
\end{equation*}
for all $X,Y,Z$: the left hand side satisfies the Bianchi identity in $X,Y,Z$
while the right hand side does not (for $n>1$), so they must vanish
independently.

(3)$\Rightarrow$(4) by equation~\eqref{eq:third} again: the extended
system~\eqref{eq:extsystem} implies $\Nu{B}(X,J\Lam)=0$.

(4)$\Rightarrow$(5) is immediate: if $\Lam$ is not identically zero, it is
nonzero on an open dense subset, because $J\Lam$ is a Killing vector field
by Lemma~\ref{lem:lambdakilling}.

(5)$\Rightarrow$(3). Given a nonzero $Z\in N(B)_p$, substitute $Y=Z$ and
$\mu=0$ in equation~\eqref{eq:second} to obtain $[K(X,Z),\nabla\Lam-BA]=0$.
Hence by Lemma~\ref{lem:Kbracket} there is a scalar $\mu=\mu(p)$ such that
$\nabla\Lam-BA=\mu\,\Id$ at $p$. This holds at every point of a dense subset for
functions $\mu,B$ defined on this subset. Moreover, $A$ is not proportional to
the identity at every point of a dense open set (this is straightforward to
show using \eqref{eq:main}---for a proof see \cite[Lemma 4]{FKMR}).  Then on a
neighbourhood $U$ of any point in this dense open set, $B$ and $\mu$ are
smooth functions (being solutions of an inhomogeneous linear system of maximal
rank with smooth coefficients). We need to show that $B$ is constant and
$\tau:=\d \mu-2Bg(\Lam,\cdot)$ is identically zero on $U$. For this, suppose
that a nonzero vector $Z$ is in the $B$-nullity of the curvature and insert
$\nabla\Lam=\mu\Id+BA$ into~\eqref{eq:third} to obtain
\begin{equation}\label{eq:third0}
J\Nu{B}(X,J\Lam)+\tau(X)\Id+\d B(X)A=0,
\end{equation}
and hence, by applying this identity to $Z$, $\tau(X)Z+\d B(X)AZ=0$. If $Z$ is
not an eigenvector of $A$, we have $\tau(X)=\d B(X)=0$ for all $X\in TU$ which
is what we wanted to show. We may thus assume $AZ=\xi Z$ for some function
$\xi$, so that $\tau=-\xi \d B$ and
\begin{equation}\label{eq:third1}\begin{split}
\Nu{B}(X,\Lam)&=\d B(JX)(A-\xi\Id)J\\
&=((A-\xi\Id)X)^\flat \otimes (\d B)^\sharp-\d B\otimes (A-\xi \Id)X,
\end{split}
\end{equation}
where $\alpha^\sharp$ denotes the metric dual of a $1$-form $\alpha$, and
the second line follows from the Bianchi symmetry satisfied by
$g(\Nu{B}(X,\Lam)Y,W)=g(\Nu{B}(Y,W)X,\Lam)$. Comparing the second and third
lines, it follows that $A-\xi\Id$ has complex rank $\leq 1$. It remains
to show the following.
\begin{lem}\label{low-rank} Suppose $A\in\Sol(g,J)$ and that on an open subset
  $U$\textup: $A$ is not parallel with exactly two \textup(distinct\textup)
  eigenvalues, both smooth, and $Z$ is an eigenvector of $A$ in the
  $B$-nullity of $g$ for smooth $B$. Then $\d B=0$ on $U$.
\end{lem}
Given this lemma, whose proof give below, we obtain also $\tau=-\xi \d B=0$ on
$U$, and hence the system \eqref{eq:extsystem} holds in a neighbourhood of
every point of an open dense subset for a (local) constant $B$ and a smooth
function $\mu$.  On the other hand, it was proven in \cite[\S 2.5]{FKMR} that
the constants $B$ are the same for each such neighbourhood. Taking the trace of
the second equation in \eqref{eq:extsystem}, we obtain $2m\mu=\trace
\nabla\Lam-B\trace (A)$, so that the functions $\mu$ coincide on overlaps and
patch together to a globally defined function. Hence the system
\eqref{eq:extsystem} holds everywhere on $M$ for a constant $B$ and a smooth
function $\mu$.

(1--5)$\Rightarrow$(6). Since $\Nu{B}(X,Y)\Lam=0$,
equation~\eqref{eq:lambda1} implies
\begin{equation}\label{eq:ZBongradients}
0=\sum_{i=1}^{l}\Nu{B}(X,Y)\grad_g \xi_i,
\end{equation}
where $\xi_1,\ldots \xi_l$ are the eigenvalues of $A$. It was shown in
\cite[Proposition 14]{ACG1} and \cite[Proposition 1]{MatRos1} that the
gradient $\grad_g \xi_i$ is contained in the eigenspace of $A$
corresponding to $\xi_i$. Since $[\Nu{B}(X,Y),A]=0$, $\Nu{B}(X,Y)$ leaves the
eigenspaces of $A$ invariant. Then wherever $\Nu{B}(X,Y)\grad_g \xi_i$
is nonzero, it is an eigenvector of $A$ corresponding to the eigenvalue
$\xi_i$ and \eqref{eq:ZBongradients} shows that
$\Nu{B}(X,Y)\grad_g \xi_i=0$.

(6)$\Rightarrow$(5). It was shown in \cite[Proposition 14]{ACG1} that
every nonconstant eigenvalue $\xi$ of $A\in\Sol(g,J)$ has
nonvanishing differential on an open and dense subset.

\medbreak
The final observation of the theorem follows because condition (5) is
independent of $A\in\Sol(g,J)$, and if $A$ is $\nabla$-parallel (i.e., the
corresponding $\Lam$ is zero), then equation~\eqref{eq:second} and
Lemma~\ref{lem:Kbracket} imply that $A$ is a multiple of the identity or
$B=0$.
\end{proof}

\begin{proof}[Proof of Lemma~\textup{\ref{low-rank}}] Since $A$ is nonparallel
(i.e., $\Lam\neq 0$), it has at least one nonconstant eigenvalue. We consider
first the case that $A$ has one nonconstant eigenvalue $\xi$ and one constant
eigenvalue, which we may assume to be zero. The $\xi$-eigenspace is
therefore spanned by $\Lam$, and if this is in the nullity, then
\eqref{eq:third0} implies $d B=0$. Thus we may assume $AZ=0$, hence
$\d\mu=2B g(\Lam,\cdot)= B\,\d\xi$, so that $\mu$ and $B$ are functions
of $\xi$.  For any $X$ with $AX=0$ we have
\begin{equation*}
\xi \mu  X = -(A-\xi\,\Id)\mu X = -(A-\xi\,\Id)\nabla_X\Lam
= (\nabla_X A-\d\xi(X))\Lam = g(\Lam,\Lam) X
\end{equation*}
since $\nabla\Lam=\mu\,\Id+B A$, $\d\xi(X)=0$ and $\nabla A$ is given
by~\eqref{eq:main}. Hence $\xi\mu=g(\Lam,\Lam)$.

On the other hand, using the Gray--O'Neill formulae~\cite{ACG1,Gray} or the
explicit form
\begin{equation*}
g=-\xi g_0 + \frac{d\xi^2}{\Theta(\xi)} + \Theta(\xi) \theta^2
\end{equation*}
of the metric, we obtain that
\begin{equation}\label{eq:curvature}
-4B K(X,Y)Z = R(X,Y)Z = R_0(X,Y)Z-\frac{4g(\Lam,\Lam)}{\xi^2}K(X,Y)Z,
\end{equation}
for $X,Y$ in the zero eigenspace of $A$, where $R_0$ denotes the curvature of
$g_0$ (the K\"ahler quotient by $J\Lam$), lifted to the zero eigenspace.  Thus
$R_0(X,Y)Z = 4(B\xi-\mu) K_0(X,Y)Z$, where $K_0=-K/\xi$ is the algebraic
constant holomorphic sectional curvature tensor of $g_0$. Taking the trace
over $X$ (on the zero eigenspace), $\mathrmsl{Ric}_0(Z)=2(m-1) (B\xi-\mu) Z$
and so $B\xi-\mu$ is independent of $\xi$, hence constant. This combines with
$\d\mu=B\,\d\xi$ to give $\d B=0$ as required.

Now we turn to the case when $A$ has two nonconstant eigenvalues $\xi_1$ and
$\xi_2$.  Note that in this case, $M$ is necessarily real $4$-dimensional. Let
$V_1=\grad_g\xi_1$, $V_2=\grad_g\xi_2$ and suppose that $V_2$ is contained in
the $B$-nullity of the curvature $R$. To compute $B$ and $\mu$, we apply $V_1$
and $V_2$ to the equation $\nabla\Lam=\mu\,\Id+BA$ to obtain the linear system
$m_1=\mu+\xi_1 B,$ $m_2=\mu+\xi_2 B$, where $m_1$, $m_2$ are the eigenvalues
of $\nabla\Lam$, i.e., $\nabla_{V_1}\Lam=m_1V_1$ and
$\nabla_{V_2}\Lam=m_2V_2$.  Hence,
\[
\mu=\frac{\xi_2m_1-\xi_1m_2}{\xi_2-\xi_1},\qquad B=\frac{m_1-m_2}{\xi_1-\xi_2}.
\]
To calculate $m_1,m_2$, we recall that $\Lam=\frac{1}{2}(V_1+V_2)$ and so
$\nabla_\Lam \Lam=\frac{1}{2}(m_1 V_1+m_2V_2)$, or, dually, $\d(g(\Lam,\Lam))
=m_1 \d\xi_1+m_2\d\xi_2$. The classification of hamiltonian $2$-forms
from~\S\ref{subsec:localclass} shows that, in a neighbourhood of almost every
point, $g$ takes the form
\[
g=\frac{\xi_1-\xi_2}{F_1(\xi_1)}\d\xi_1^2+\frac{\xi_2-\xi_1}{F_2(\xi_2)}\d\xi_2^2
+\frac{F_1(\xi_1)}{\xi_1-\xi_2}(\d t_1+\xi_2 \d t_2)^2
+\frac{F_2(\xi_2)}{\xi_2-\xi_1}(\d t_1+\xi_1 \d t_2)^2
\]
in local coordinates $\xi_1,\xi_2,t_1,t_2$.  From this, we obtain
$g(\Lam,\Lam)$ in terms of the functions $F_1,F_2$. Calculating
$\d(g(\Lam,\Lam))$ and comparing coefficients, we obtain
\begin{align}\label{eq:Bexplicit}
B=\frac{m_1-m_2}{\xi_1-\xi_2}
=\frac{(F_1'(\xi_1)+F_2'(\xi_2))(\xi_1-\xi_2)-2(F_1(\xi_1)-F_2(\xi_2))}
{4(\xi_1-\xi_2)^3}.
\end{align}
Replacing $X$ in \eqref{eq:third1} by the vector $JV_2$ in the nullity, we see
that $\d B(V_2)=0$, i.e., $B$ does not depend on the variable $\xi_2$.  Using
\eqref{eq:Bexplicit}, it is straightforward to show that the condition $\d
B/\d\xi_2=0$ is equivalent to
\begin{align}\label{eq:derivB}
0=F_2''(\xi_2)(\xi_1-\xi_2)^2+2(F_1'(\xi_1)+2F_2'(\xi_2))(\xi_1-\xi_2)
-6(F_1(\xi_1)-F_2(\xi_2)).
\end{align}
Taking three derivatives of this equation w.r.t.\ $\xi_1$ yields
$F_1^{(4)}(\xi_1)=0$, hence $F_1(\xi_1)$ is a polynomial of degree $\leq 3$.
Inserting this condition back into \eqref{eq:derivB}, a straightforward
calculation shows $F_1=F_2$. Inserting these polynomials into
\eqref{eq:Bexplicit} shows that $B$ is a constant.  This also follows
from~\cite{ACG1} where it is shown that $(g,J)$ has constant holomorphic
sectional curvature (and hence $B$ is constant) if $F_1=F_2$ is a polynomial
of degree $\leq 3$.
\end{proof}

\begin{rem}
Recall from Remark \ref{rem:affine} that $A$ not being parallel is necessary
for $(5)$. All other conditions are automatically fulfilled for parallel $A$,
in which case we have $\mu=B=0$.
\end{rem}
To relate this result to the local classification of metrics with hamiltonian
$2$-forms (see \S\ref{subsec:localclass}), observe that at each point in a
dense open set, the $J$-linear span of the gradients of the eigenvalues of $A$
is the tangent space to the orthotoric fibres of the metric $g$.

\begin{cor}\label{cor:CHSC-fibre} $A\in\Sol(g,J)$ satisfies the extended
system~\eqref{eq:extsystem} for $B\in\R$ if and only if $A$ is parallel
\textup(in which case, we may assume $B=0$\textup) or the orthotoric fibres of
$A$ are in the $B$-nullity of $g$. In particular, since these fibres are
totally geodesic, they have constant holomorphic sectional curvature $-4B$.
\end{cor}

By Definition~\ref{defn:CCB}, a K\"ahler metric is $\CC(B)$ for a constant $B$
if one of the conditions in Theorem \ref{thm:equivalence} is satisfied for
some non-parallel $A\in \Sol(g,J)$.  We next describe the conditions on the
parameters in formula~\eqref{eq:locclass} under which a K\"ahler metric is
$\CC(B)$.

For this, we first observe that the extended system~\eqref{eq:extsystem} is
equivalent to the special case $\tau_0=0$, $\tau_1=-4B$, $\tau_2=-4\mu$ of the
system~\cite[\S2.3, Equation (30)]{ACG1}, with the term ``$W^{\mathcal
  K}(\phi)$'' omitted.  Hence (the proof of)~\cite[Proposition 5]{ACG1}
applies to show that the polynomial
\begin{equation}\label{eq:inv-poly}
F(t) = -4 (Bt+\mu) p_A(t) - g(K,K(t)),
\end{equation}
has constant coefficients, where $p_A(t)$ is the characteristic polynomial of
$A$, $K=J\grad_g\sigma_1$ and $K(t)=J\grad_g p_A(t)$ (thus $K$ coincides with
the Killing vector field $2J\Lam$). To interpret this fact geometrically, we
next observe that any triple $(A,\Lam,\mu)\in \mathfrak{gl}(TM)\oplus TM\oplus
M\times\R$, with $A$ hermitian, defines a hermitian (bundle) metric on
$\{(\sigma,\rho)\in\Hom(TM,\C)\oplus M\times\C: \sigma(JX)=i\sigma(X)\}$, via
the expression
\begin{equation}
\begin{bmatrix} \rho \\ \sigma\end{bmatrix}^\dag
\begin{bmatrix} \mu & \Lam \\ \Lam & A \end{bmatrix}
\begin{bmatrix} \rho' \\ \sigma' \end{bmatrix}
:=\mu\overline\rho\rho' +\overline\sigma(\Lam)\rho'+\overline\rho\sigma'(\Lam)
+g(\overline \sigma\circ A, \sigma').
\end{equation}
When $B=-1$, this bundle may be identified with the (holomorphic) tangent
bundle of the complex cone over $(M,g,J)$ studied in~\cite{MatRos2}, which we
discuss in Appendix~\ref{sec:cone}.  For any $B\in\R$, the bundle carries a
connection $\cD$ defined by
\begin{equation}\label{eq:tractor}
\cD_X \begin{bmatrix} \rho \\ \sigma \end{bmatrix}
= \begin{bmatrix} \nabla_X\rho +\sigma(X) \\ \nabla_X \sigma
+Bg(X+iJX,\cdot)\rho \end{bmatrix}. 
\end{equation}
This connection induces the extended system~\eqref{eq:extsystem} in the
following sense (cf.~\cite[\S\S 4.1-4.2]{FKMR} in the case $B\neq 0$).
\begin{lem} For any sections $(A,\Lam,\mu)$ and $(\sigma,\rho)$ as above,
we have
\begin{multline*}
\partial_X \biggl(\begin{bmatrix} \rho \\ \sigma\end{bmatrix}^\dag
\begin{bmatrix} \mu & \Lam \\ \Lam & A \end{bmatrix}
\begin{bmatrix} \rho' \\ \sigma' \end{bmatrix}\biggr)
-\biggl(\cD_X\begin{bmatrix} \rho \\ \sigma\end{bmatrix}\biggr)^\dag
\begin{bmatrix} \mu & \Lam \\ \Lam & A \end{bmatrix}
\begin{bmatrix} \rho' \\ \sigma' \end{bmatrix}
-\begin{bmatrix} \rho \\ \sigma\end{bmatrix}^\dag
\begin{bmatrix} \mu & \Lam \\ \Lam & A \end{bmatrix}
\cD_X\begin{bmatrix} \rho' \\ \sigma' \end{bmatrix}\\
=\begin{bmatrix} \rho \\ \sigma\end{bmatrix}^\dag
\begin{bmatrix} \nabla_X\mu-2B g(\Lam,X) & \nabla_X\Lam-\mu X- BAX \\
\nabla_X\Lam-\mu X- BAX & \nabla_X A -\bigl(\begin{smallmatrix}
X^\flat\otimes\Lam+\Lam^\flat\otimes X\qquad\\
\;+JX^\flat\otimes J\Lam+J\Lam^\flat\otimes JX)\end{smallmatrix}\bigr)
\end{bmatrix}\begin{bmatrix} \rho' \\ \sigma' \end{bmatrix}
\end{multline*}
\end{lem}
\noindent The proof is a straightforward computation. Up to a normalization
constant, the function $F(t)$ is the (complex) determinant of the hermitian
form on $\Hom(TM,\C)\oplus M\times\C$ defined by $(A-t\,\Id,\Lam,\mu+Bt)$, so
its roots are the relative eigenvalues of the hermitian forms defined by
$(A,\Lam,\mu)$ and $(\Id,0,-B)$. This gives another proof that $F(t)$ has
constant coefficients when $(A,\Lam,\mu)$ solves~\eqref{eq:extsystem}, and
further shows that the relative eigenspaces are $\cD$-parallel subbundles of
$\Hom(TM,\C)\oplus M\times\C$.

\begin{thm}\label{thm:condforCCB} Let $(g,J,\omega)$ be a K\"ahler metric with
a non-parallel hamiltonian $2$-form, given explicitly by~\eqref{eq:locclass}
on a dense open set.  Then $g$ is $\CC(B)$ if and only if
$\Theta_j(t)=\Theta(t)$, a polynomial of degree $\leq\ell+1$
\textup(independent of $j$\textup) with leading coefficient $-4B$, and
$\Theta(\eta)=0$ for all constant eigenvalues $\eta$.
\end{thm}
\begin{proof} If $g$ is $\CC(B)$ then the (totally geodesic) orthotoric fibres
have constant holomorphic sectional curvature (CHSC). The hamiltonian $2$-form
restricts to a hamiltonian $2$-form on each fibre whose characteristic
polynomial is $p_A(t)/p_c(t)$. Applying~\cite[Proposition 18]{ACG1} fibrewise,
we thus have $\Theta_j(t)=\Theta(t):=F(t)/p_c(t)$ for all $j$ (where we recall
that $p_c(t)$ is the monic polynomial whose roots are the constant eigenvalues
$\eta$ of $A$). It remains to show that any root $\eta$ of $p_c(t)$ is a root
of $\Theta(t)$, i.e., the multiplicity of $\eta$ as a root of $F(t)$ is
greater than its multiplicity as a root of $p_c(t)$. The latter is the
dimension of the kernel of $A-\eta\,\Id$ in $\Hom(TM,\C)$ which is a subspace
$U$ of the relative $\eta$-eigenspace, i.e., the kernel of the hermitian form
defined by $(A-\eta\,\Id,\Lam,\mu+B\eta)$. However by~\eqref{eq:tractor}, $U$
cannot be $\cD$-parallel, so the dimension of the relative $\eta$-eigenspace
is strictly larger, hence so is the multiplicity of $\eta$ as a root of
$F(t)$.

Conversely, if $\Theta_j(t)=\Theta(t)$ as stated, then the orthotoric fibres
belong to the $B$-nullity of $g$. To see this, observe that the Gray--O'Neill
curvature formulae~\cite{Gray,ONeill} (with Gray--O'Neill tensor
\eqref{eq:GrayONeill}) imply that all components of the curvature of $g$,
apart from the purely horizontal part, depend on the base metrics $g_\eta$ in
\eqref{eq:locclass} only to first order at each point.  Hence, to compute
$R(X,Y)Z$ for $Z$ vertical, we may use a metric $\tilde g$ which agrees with
$g$ at a given point, but where we replace the base metrics $g_\eta$ with
metrics $\tilde g_\eta$ which have CHSC equal to $\Theta'(\eta)$ at that
point. By~\cite[Proposition 17]{ACG1}, $\tilde g$ has CHSC given by a multiple
of $B$, hence the fibres are in the $B$-nullity. Consequently the same holds
for $g$.
\end{proof}
\begin{cor}\label{cor:degree} Let $(g,J)$ be a $\CC(B)$ K\"ahler metric
\textup(e.g., with $D(g,J)\geq 3$\textup) which is weakly Bochner-flat
\textup(or is Bochner-flat\textup).  Then $g$ is K\"ahler--Einstein
\textup(or has constant holomorphic sectional curvature, respectively\textup).
\end{cor}
Recall from \cite{ACG1} that a K\"ahler metric $(g,J)$ of dimension $2m$ is
\emph{orthotoric} if it admits a hamiltonian $2$-form having $m$ nonconstant
eigenvalues $\xi_1,\ldots\xi_m$ (these metrics are also
``K\"ahler--Liouville''---see~\cite{KiyTop}).
\begin{cor}\label{cor:degree0} Let $(g,J)$ be a  $\CC(B)$ K\"ahler metric
\textup(e.g., with $D(g,J)\geq 3$\textup) which is orthotoric. Then $g$ has
constant holomorphic sectional curvature.
\end{cor}
\begin{rem}
An analog of the corollary in real projective geometry, which is also true
under more general assumptions, can be found in \cite{Fubini}.
\end{rem}

Theorem~\ref{thm:condforCCB} has the following global consequence.

\begin{thm}\label{thm:orbifold} Let $M$ be a closed connected $2m$-orbifold
  \textup($2m\geq 4$\textup) and suppose $(g,J)$ is a $\CC(B)$ K\"ahler metric
  on $M$. Then $(M,g,J)$ is an orbifold quotient of $\C P^m$ with a
  Fubini--Study metric.
\end{thm}
\begin{proof} By assumption, $M$ admits a hamiltonian $2$-form of order
$\ell\geq 1$. The theory of~\cite[Section 2]{ACGT2}, which extends to
orbifolds following~\cite{LT}, shows that the universal orbifold cover of $M$
has a blow-up $\hat M$ which is a bundle of (connected) toric orbifolds over
an orbifold $S$ which is a complete K\"ahler product over the constant
eigenvalues $\eta$ of $A$. Since blow-up does not change the orbifold
fundamental group, $\hat M$ is a simply connected orbifold, hence so is $S$
(since the fibres of $\hat M\to S$ are connected).  Now, since every constant
eigenvalue $\eta$ is a root of the function $\Theta$ of
Theorem~\ref{thm:condforCCB}, it follows from~\cite[Proposition 6]{ACGT2} (or
rather, its proof, extended straightforwardly to orbifolds) that $S$ is a
K\"ahler product of complex projective spaces where the K\"ahler metric on the
factor corresponding to a root $\eta$ has CHSC $-\Theta'(\eta)$
(see~\cite[Theorem 5(iv--v)]{ACGT2}).  As discussed
in~\S\ref{subsec:localclass}(1), these are precisely the conditions (given
that $\Theta$ is a polynomial of degree $\leq\ell+1$ vanishing on the constant
eigenvalues $\eta$) which ensure that the metric on $M$ has
CHSC~\cite{ACG1}. (This is not a coincidence: the Fubini--Study metric on $\C
P^m$ admits hamiltonian $2$-forms of any order $0\leq \ell\leq m$.)  Since
$K=J\Lam$ is a nonparallel Killing vector field on $M$ (it is hamiltonian,
hence has zeros), the curvature of $g$ must be positive by Bochner's argument.
Hence the universal cover of $M$ is isometric to $\C P^m$ with a Fubini--Study
(positive CHSC) metric.
\end{proof}
\begin{cor} Let $(M,g,J)$ be a closed connected K\"ahler orbifold of
dimension $2m\geq 4$ and mobility $D(g,J)\geq 3$.  Then either $M$ is an
orbifold quotient of $\C P^m$ with a Fubini--Study metric, or every K\"ahler
metric c-projectively equivalent to $g$ is affinely equivalent to $g$.
\end{cor}
\begin{rem} This corollary is immediate from Theorem~\ref{thm:orbifold} and
Theorem~\ref{thm:extsystem} (i.e.,~\cite[Section 2]{FKMR}): in the manifold
case, it is the main result of~\cite{FKMR}, where it was established for
metrics of arbitrary signature. Indeed, on manifolds, the analogue of
Theorem~\ref{thm:orbifold} for metrics of arbitrary signature was obtained
in~\cite[Remark 12]{FKMR}. Furthermore, the proof in~\cite{FKMR} proceeds by
first reducing to the case that $-B g$ is positive definite, and this part of
the argument extends straightforwardly to orbifolds. Hence
Theorem~\ref{thm:orbifold} is actually valid in all signatures. 

On the other hand, in the remaining case, where (without loss of generality)
$B=-1$ and $g$ is positive definite, \cite[Lemma 8]{FKMR} shows that the
extended system~\eqref{eq:extsystem} yields a nontrivial solution of the
k\"ahlerian \emph{Tanno equation}, and so the manifold case of
Theorem~\ref{thm:orbifold} follows from \cite[Theorem 10.1]{tanno}.  In fact,
as shown in \cite[\S 3, see (4)]{FR}, the Tanno equation is equivalent to the
extended system in this case, and so Theorem~\ref{thm:orbifold} may be
regarded as providing a natural generalization of \cite[Theorem 10.1]{tanno}
to orbifolds of arbitrary signature.  Note that our method of proof for
Theorem~\ref{thm:orbifold} is very different from~\cite{tanno}.

The corollary is a rigidity result for closed connected K\"ahler orbifolds
$(M,g,J)$ which are not quotients of $\C P^m$, but admit a c-projectively
equivalent metric which is not affine equivalent (i.e., a hamiltonian $2$-form
of order $\ell>0$). This has several consequences. First, as observed
in~\cite{FKMR}, the isometry group of $g$ has codimension $\leq 1$ in the
group of c-projective transformations of $M$: this is because the latter group
acts on the projectivization of $\Sol(M,g)$, with the isometry group of $g$ as
a point stabilizer. Secondly, since the hamiltonian $2$-form is essentially
unique (i.e., $A\in\Sol(M,g)$ is unique up to a linear combination with the
identity solution), the $\ell$-torus action it defines must be central.
\end{rem}

\section{Classification of metrics with c-projective mobility
$\geq 3$}
\label{sec:applications}

Let us recall the following result:
\begin{thm}[\cite{ACGT2,FKMR}]\label{thm:dim4}
Let $(M,g,J)$ be a connected K\"ahler manifold of real dimension $4$. Then
$D(g,J)\geq 3$ if and only if the holomorphic sectional curvature is constant.
\end{thm}
\begin{rem}
In \cite[Proposition 10]{ACGT2} and \cite[Lemma 7]{FKMR}, it was shown that a
K\"ahler manifold of real dimension $4$ and of mobility $\geq 3$ has constant
holomorphic sectional curvature. The fact that every CHSC K\"ahler manifold of
any dimension $2m$ has mobility $(m+1)^2\geq 3$ is a standard result, see for
example~\cite{ACG1,Mikes}.
\end{rem}

By Theorem \ref{thm:extsystem}, the condition $D(g,J)\geq 3$ implies either
that all $A\in\Sol(M,g)$ are parallel, or that the metric is $\CC(B)$, i.e.,
the equivalent conditions of Theorem~\ref{thm:equivalence} hold.  Conversely,
we now find the metrics satisfying $D(g,J)\geq 3$ among those that are
$\CC(B)$.
\begin{thm}\label{thm:Dgeq3}
Let $(M,g,J)$ be a connected K\"ahler manifold of real dimension $2m\geq 4$
which is $\CC(B)$. Suppose in addition that there exists $A\in\Sol(g,J)$ such
that
\begin{itemize}
\item either the number of nonconstant eigenvalues of $A$ is $\geq 2$
\item or the number of constant eigenvalues of $A$ is $\geq 3$.
\end{itemize}
Then $D(g,J)\geq 3$.
\end{thm}
\begin{proof}
Let us choose $A\in\Sol(g,J)$ satisfying one of the two conditions on the
eigenvalues.

First suppose that the corresponding vector field $\Lam$ is identically
zero.  Then $A$ is covariantly constant and all eigenvalues of $A$ are
constant (see Remark \ref{rem:affine}).  The endomorphism $\tilde{A}=A^2$ is
covariantly constant and hence contained in $\Sol(g,J)$. It follows that
$\tilde{A},A$ and $\Id$ are linearly independent and therefore $D(g,J)\geq 3$,
since otherwise, $A$ would be annihilated by a polynomial with constant
coefficients of order two or lower and this contradicts the assumption that
the number of constant eigenvalues is at least three. We have proven Theorem
\ref{thm:Dgeq3} under the assumption $\Lam\equiv 0$.

Let us now suppose that $\Lam$ is not identically zero. 

\emph{First case: $B=0$.} A straightforward computation (using the equations
in \eqref{eq:extsystem}) shows
\begin{equation*}
\tilde{A}=\Lam^\flat\otimes\Lam+J\Lam^\flat \otimes J\Lam
\end{equation*}
is contained in $\Sol(g,J)$, where the corresponding vector field is
$\tilde{\Lam}=\mu\Lam$ and $\mu$ is a constant.

Clearly, $\tilde A$ is not proportional to $\Id$ (since it is multiplication
with $g(\Lam,\Lam)$ on $\mathrm{span}\{\Lam,J\Lam\}$ and multiplication with
$0$ on $\mathrm{span}\{\Lam,J\Lam\}^\perp$).  If $D(g,J)=2$, we have $A=\alpha
\tilde A+\beta\Id$ for certain constants $\alpha$ and $\beta$ but this
contradicts the assumptions on the eigenvalues of $A$.  Theorem
\ref{thm:Dgeq3} is proven in the case $B=0$.  \medbreak \emph{Second case:
  $B\neq 0$.} Let us multiply the metric with $-B$, such that the system
\eqref{eq:extsystem} for the new metric (which we again denote by the symbol
$g$) holds with $B=-1$. Note that the mobility remains unchanged by this
procedure. A straightforward computation (one may also compare
\cite[p. 1338]{Mikes}, \cite[equation (88) in the proof of Lemma 10]{FKMR} or
the cone construction \cite[Theorem 9]{MatRos2}---see the appendix below)
using the equations in \eqref{eq:extsystem} shows
\begin{equation*}
\tilde{A}=A^2+\Lam^\flat\otimes\Lam+J\Lam^\flat\otimes J\Lam
\end{equation*}
is contained in $\Sol(g,J)$ with corresponding vector field
$\tilde{\Lam}=(A+\mu\Id)\Lam$. Assuming $D(g,J)=2$, we obtain (up to
rescaling) $A=\tilde A+\alpha\Id$ for a certain constant $\alpha$. Taking the
covariant derivative of this equation shows
$\Lam=(A+\mu\Id)\Lam$. Hence, $\Lam$ is an eigenvector of $A$
corresponding to the nonconstant eigenvalue $1-\mu$. Equation
\eqref{eq:lambda1} (together with the fact that for each nonconstant
eigenvalue $\xi_i$ of $A$, $\grad_g\xi_i$ is contained in the corresponding
eigenspace) implies that $A$ has exactly one nonconstant eigenvalue.
Restricting $A=\tilde A+\alpha\Id$ to the orthogonal complement
$U:=\mathrm{span}\{\Lam,J\Lam\}^{\perp}$ shows that the restriction
$A|_{U}$ is annihilated by a quadratic polynomial. Then the number of
nonconstant eigenvalues is at most two.  We obtain a contradiction to any of
the two conditions on the eigenvalues of $A$. Hence, $D(g,J)\geq 3$ and
Theorem \ref{thm:Dgeq3} is proven.
\end{proof}

\appendix
\section{Cone construction for $\CC(-1)$ metrics}
\label{sec:cone}

\subsection{The cone construction}
\label{subsec:cone}

If $g$ is a $\CC(-1)$ K\"ahler metric then the space $\Sol(g,J)$ is isomorphic
to the space of solutions $(A,\Lam,\mu)$ of the PDE system
\begin{equation}\label{eq:B=-1system}\begin{split}
\nabla_X A&=X^\flat\otimes \Lam+\Lam^\flat\otimes X+JX^\flat\otimes J\Lam
+J\Lam^\flat\otimes JX,\\
\nabla\Lam&=\mu \Id-A,\\
\nabla \mu&=-2\Lam^\flat.
\end{split}\end{equation}
The cone construction \cite[Theorem 9]{MatRos2} (see also the formulae in
\cite[pp.~1338--1339]{Mikes} for the same statement, though the formula for
$\hat A$ appearing there seems to have a misprint) asserts that the space of
solutions $(A,\Lam,\mu)$ of this system is isomorphic to the space of parallel
hermitian endomorphisms $\hat A\in \mathrm{End}(T\hat M)$ on the cone
\begin{align}\label{eq:cone}
\hat M=\R_{>0}\times\R\times M,\,\,\,\hat g=\d r^2+r^2(\phi^2+g),\,\,\,
\hat J=\frac{1}{r}\partial_t\otimes \d r-r\partial_r\otimes \phi+J,
\end{align}
where $\phi=\d t-\tau$ and $\tau $ is a $1$-form on $M$ satisfying $\d
\tau=2\omega$ ($\omega=g(J\cdot,\cdot)$ denotes the K\"ahler form on $M$). The
construction is local but this is sufficient for our purposes. The
correspondence between solutions $(A,\Lam,\mu)$ of \eqref{eq:B=-1system}
and parallel hermitian endomorphisms $\hat A\in \mathrm{End}(T\hat M)$ is
given by
\begin{align}\label{eq:hatA}
\hat g(\hat{A}\cdot,\cdot)=\mu\d r^2-r\d r\odot\Lam^\flat
+r^2(\mu\phi^2+\phi\odot\Lam^\flat(J\cdot)+g(A\cdot,\cdot)).
\end{align}

Further, we view the manifold $N=\R\times M$ with metric $h=\phi^2+g$ as
naturally embedded into $\hat M$ as the hypersurface $N=\{r=1\}$. The manifold
$(M,g,J)$ is recovered from $(\hat M,\hat g,\hat J)$ as the K\"ahler quotient
w.r.t.\ the action of the hamiltonian Killing vector field $K:=\frac{1}{2}\hat
J\grad_{\hat g}r^2$ on the level set $N$, where the function
$\frac{1}{2}r^2$ serves as the moment map for the (local) hamiltonian
$S^1$-action induced by $K$.

\subsection{The K\"ahler quotient in the presence of a decomposition
of the cone into a direct product} \label{subsec:quotient}

By the decomposition theorem for riemannian manifolds \cite{ei}, the
parallel hermitian endomorphisms on a manifold are classified by all the ways
the manifold can be decomposed into a direct product of K\"ahler manifolds.
Let $(\hat M,\hat g,\hat J)$ be the cone over a K\"ahler manifold $(M,g,J)$
given by~\eqref{eq:cone}. Suppose $\hat g$ decomposes into a direct product
\begin{align}\label{eq:directprod}
M=\prod_i M_i,\,\,\,\hat g=\sum_i \hat g_i,\,\,\,\hat J=\sum_i \hat J_i.
\end{align}
of K\"ahler manifolds $(\hat M_i,\hat g_i,\hat J_i)$. Recall that the cone
structure on $(\hat M,\hat g,\hat J)$ gives rise to the cone vector
field $\cC=r\partial_r$ satisfying
$\hat\nabla\cC=\Id$. Conversely, a vector field satisfying this
equation induces a cone structure by defining the radial coordinate to be
\begin{equation*}
r:=\sqrt{\hat g(\cC,\cC)}.
\end{equation*}
The decomposition $\cC=\sum_{i=0}^\ell \cC_i$ of the cone
vector field w.r.t.~\eqref{eq:directprod} defines cone vector fields
$\cC_i$ on each component $(\hat M_i,\hat g_i,\hat J_i)$ making them
into cones over certain K\"ahler manifolds $(M_i,g_i,J_i)$. Hence, having a
decomposition as in \eqref{eq:directprod}, we may write
\begin{align}\label{eq:eigdecomphatg}
\hat g=\sum_{i=0}^\ell \underbrace{(\d r_i^2+r_i^2(\phi_i^2+g_i)}_{=\hat g_i}.
\end{align}
Here we allow some of the $g_i$'s to be zero, meaning that the corresponding
cone $(\hat M_i,\hat g_i,\hat J_i)$ is (complex) $1$-dimensional over a base
of dimension $0$. In particular $\hat g_i$ is flat.

As a K\"ahler riemannian cone, $\hat g$ is of the form \eqref{eq:cone}. Using
$\cC=\sum_{i=0}^\ell \cC_i$ and
$\cC_i=r_i\partial_{r_i}$, we see that
\[
r,K=\frac{1}{2}\hat J\grad_{\hat g}r^2,\partial_r,\d r\mbox{ and }\phi
\]
relate to the corresponding objects on the components $\hat g_i$ of $\hat g$ in
\eqref{eq:eigdecomphatg} by the equations
\begin{align}\label{eq:rel1}
r^2=\sum_{i}  r_i^2,\,\,\,K=\sum_i  K_i,\,\,\,
\partial_r=\frac{1}{r}\sum_i  r_i\partial_{ r_i},\,\,\,
\d r=\frac{1}{r}\sum_i r_i\d r_i\mbox{ and }
\phi=\frac{1}{r^2}\sum_{i} r_i^2\phi_i.
\end{align}

Next we describe the K\"ahler quotient of the direct product metric $\hat g$
in \eqref{eq:eigdecomphatg} w.r.t.\ the action of the hamiltonian Killing
vector field $K:=\frac{1}{2}\hat J\grad_{\hat g}r^2$ on the level set $r=1$.

\begin{thm}\label{thm:quotient}
The K\"ahler quotient metric $g$ of the metric $\hat g$ is given by the formula
\begin{align}\label{eq:quotient}
g=\sum_{i=0}^\ell \d r_i^2+\frac{1}{2}\sum_{i,j=0}^\ell r_i^2
r_j^2(\phi_i-\phi_j)^2+\sum_{i=0}^\ell r_i^2 g_i.
\end{align}
\end{thm}

\begin{rem}
The forms $\phi_i-\phi_j$ are basic, i.e., they can be written as the pullback
of forms defined on the quotient. Indeed, these forms vanish upon insertion of
$\partial_r$ and $K$, they do not depend on $r$ and they are $K$-invariant
(that is, invariant w.r.t.\ the (local) $S^1$-action).
\end{rem}

\begin{rem}
Recall that the metrics $g_i$ in \eqref{eq:quotient} are zero if $\hat g_i=\d
r_i^2+r_i^2(\phi_i^2+g_i)$ is (complex) $1$-dimensional.
\end{rem}

\begin{proof}[Proof of Theorem \ref{thm:quotient}]
Restricted to the level set $r=1$, the quotient metric $g$ is given by 
\begin{align}\label{eq:groughly}
g=\hat g-\phi^2=\sum_{i=0}^\ell \d r_i^2
+\sum_{i=0}^\ell r_i^2\phi_i^2-\phi^2+\sum_{i=0}^\ell r_i^2 g_i.
\end{align}
Using \eqref{eq:rel1}, we obtain
\begin{equation*}
\sum_{i=0}^\ell r_i^2\phi_i^2-\phi^2
=\sum_{i=0}^\ell r_i^2\phi_i^2-\sum_{i,j=0}^\ell r^2_i r_j^2\phi_i\otimes  \phi_j
=\frac{1}{2}\sum_{i,j=0}^\ell r_i^2 r_j^2(\phi_i-\phi_j)^2
\end{equation*}
which gives us formula \eqref{eq:quotient}.
\end{proof}

In what follows, let $\hat A$ be a parallel hermitian endomorphism for $\hat
g$ with distinct eigenvalues $C_0<\cdots<C_\ell$ of multiplicities $m_0,\ldots
m_\ell$. Let \eqref{eq:eigdecomphatg} be the decomposition of $\hat g$
w.r.t.\ the parallel eigenspace distributions of $\hat A$. If we consider
$\hat A$ as a parallel symmetric $(0,2)$-tensor field (by lowering one index
w.r.t.\ the metric $\hat g$), it is given by the formula
\begin{align}\label{eq:DecompA}
\hat A=\sum_{I=0}^\ell C_I(\d r_I^2+r_I^2(\phi_I^2+h_I)).
\end{align}

Let us relate the (constant) eigenvalues $C_0,\ldots C_\ell$ of $\hat A$ to the
(generically nonconstant) eigenvalues $\xi_1,\ldots \xi_\ell$ of $A\in
\Sol(g,J)$ corresponding to $\hat A$. 
\begin{lem}\label{lem:eigvalrel}
Let $\hat A$ be given by \eqref{eq:DecompA} for numbers $C_0<\cdots<C_\ell$ and
let the cone metric $\hat g$ over $g$ be given by
\eqref{eq:eigdecomphatg}. Let $A\in \Sol(g,J)$ correspond to $\hat A$ via the
isomorphism \eqref{eq:hatA}. Then the function $p_A\colon\hat M\times
\R\rightarrow \R$, given by
\begin{align}\label{eq:eigvalrel}
p_A(t) = \frac{1}{r^2}\prod_{i=0}^\ell(t-C_i)^{m_i-1}
\sum_{i=0}^\ell r_i^2\prod_{j\neq i}(t-C_j)
\end{align}
is the characteristic polynomial of $A$. Moreover, we have
\begin{align}\label{eq:ordering}
C_0\leq \xi_1\leq  C_1\leq\cdots\leq \xi_\ell\leq  C_\ell.
\end{align}
where $\xi_i$ are the ordered nonconstant eigenvalues of $A$. In particular,
$\ell$ is the number of nonconstant eigenvalues of $A\in \Sol(g,J)$ on the base
$M$ and the
eigenvalues of $\hat A$ occurring with multiplicity two or higher are the
constant eigenvalues of $A$.
\end{lem}

\begin{rem}
The calculations in the proof of Lemma \ref{lem:eigvalrel} below are analogous
to the derivation of elliptic separation coordinates on the $n$-sphere, see
\cite[Section 7]{Schoebel}.
\end{rem}

\begin{proof}
Recall that $A$ is the horizontal part of $\hat A$. Its action on the
horizontal distribution
\begin{equation*}
\mathcal{H}=\{X\in T \hat M:\hat g(X,\partial_r)=\hat g(X,\hat J\partial_r)=0\}
\end{equation*}
is then given by 
\begin{equation*}
AX=\hat A X-\hat g(\hat A X,\partial_r)\partial_r
-\hat g(\hat A X,\hat J \partial_r)\hat J \partial_r.
\end{equation*}
In particular, if $\xi$ is an eigenvalue of $A$, i.e., $AX=\xi X$ for some
nonzero $X\in \mathcal{H}$, we have $(\hat A-\xi \Id) X=\langle\hat A
X,\partial_r\rangle\partial_r,$ where $\langle\cdot,\cdot\rangle=\hat g-i\hat
g(\hat J\cdot,\cdot)$ denotes the hermitian inner product associated to $\hat
g$. Thus, $\xi$ is an eigenvalue of $A$ if and only if there exists $X\neq 0$
such that
\begin{equation*}
\langle X,\partial_r\rangle =0\mbox{ and }
(\hat A-\xi \Id) X=c\partial_r\mbox{ for some }c\in \C.
\end{equation*}
If $\xi$ is not an eigenvalue of $\hat A$, this condition is equivalent to
$\langle (\hat A-\xi \Id)^{-1}\partial_r,\partial_r\rangle =0.$ Inserting
$\hat A$ given by \eqref{eq:DecompA} and $\partial_r=\sum_{i=0}^\ell
\frac{r_i}{r}\partial_{r_i}$, this equation becomes equal to
\begin{align}\label{eq:eigvalrel1}
\sum_{i=0}^\ell\frac{r_i^2}{C_i-\xi}=0.
\end{align}
We obtain that each eigenvalue $\xi$ of $A$ which is not an eigenvalue of
$\hat A$ must be a solution to this equation. For fixed $r_0,\ldots r_\ell$,
the function $h(\xi)=\sum_{i=0}^\ell\frac{r_i^2}{C_i-\xi}$ has $\ell +1$ poles
at $C_0,\ldots C_\ell$ and is monotonously increasing within the intervals
$(C_i,C_{i+1})$. Hence, it has $\ell$ zeros $\xi_1,\ldots\xi_\ell$ which are the
$\ell$ nonconstant eigenvalues of $A$ depending on $r_0,\ldots r_\ell$. We have
just seen that these eigenvalues have to satisfy the relation
\eqref{eq:ordering}.

On the other hand, if an eigenvalue $C_i$ of $\hat A$ has multiplicity
$m_i\geq 2$, the corresponding eigenspace must have an $m_i-1$ dimensional
intersection with $\mathcal{H}$, hence, $C_i$ is also a constant eigenvalue of
$A$ of multiplicity $m_i-1$. The number of eigenvalues of $A$ found so far is
\begin{equation*}
\ell+\sum_{i=0}^\ell(m_i-1)=-1+\sum_{i=0}^\ell m_i=-1+\dim \hat M=\dim M.
\end{equation*}  
Thus, we certainly found all eigenvalues of $A$. 

Multiplying \eqref{eq:eigvalrel1} with $\prod_{i=0}^\ell(C_i-\xi)$, we obtain
$\sum_{i=0}^\ell r_i^2\prod_{j\neq i}(C_j-\xi)=0$. The left hand side is a
polynomial in $\xi$ of degree $\ell$ and since the nonconstant
eigenvalues $\xi_1,\ldots\xi_\ell$ are the roots of this polynomial, we obtain
\begin{equation*}
p_{\mathrm{nc}}(t)=\frac{1}{r^2}\sum_{i=0}^\ell r_i^2\prod_{j\neq i}(t-C_j),
\end{equation*}
where $p_{\mathrm{nc}}(t)=\prod_{i=1}^\ell(t-\xi_i)$ is the nonconstant part
of the characteristic polynomial of $A$. The characteristic polynomial of $A$
is then given by formula \eqref{eq:eigvalrel}.
\end{proof}

Denote by $\xi_1,\ldots\xi_\ell$ the nonconstant eigenvalues of $A$ and by
$\eta$ its constant eigenvalues of multiplicity $m_\eta$. The characteristic
polynomial $p_A(t)$, expressed in terms of the radial coordinates $r_i$, is
given by \eqref{eq:eigvalrel}; hence, we obtain the relation
\begin{align}\label{eq:eigvalrel2}
\prod_{i=1}^\ell (t-\xi_i)=\frac{1}{r^2}\sum_{I=0}^\ell  r_I^2\prod_{J\neq I}(t-C_J),
\end{align}
between the two sets of functions $\{\xi_1,\ldots\xi_\ell\}$ and $\{r_0,\ldots
r_\ell\}$. Inserting $t=C_I$ into formula \eqref{eq:eigvalrel2}, we obtain the
functions $r_I$ explicitly as functions of the $\xi_i$:
\begin{align}\label{eq:Trafo_r}
r_I^2=\frac{\prod_{i=1}^\ell (C_I-\xi_i)}{\prod_{J\neq I}(C_I-C_J)}.
\end{align}
Differentiating yields
\begin{align}\label{eq:Trafo_dr}
2r_I\d r_I
=-\sum_{i=1}^\ell\frac{\prod_{j\neq i}(C_I-\xi_j)}{\prod_{J\neq I}(C_I-C_J)}\d\xi_i.
\end{align}

\subsection{A local description of $\CC(-1)$-metrics}
\label{subsec:locclass_rederived}

We rederive the part of Theorem \ref{thm:condforCCB} stating necessary
conditions on the parameters from formula \eqref{eq:locclass} for $g$ being
$\CC(-1)$.
\begin{prop}\label{prop:nec}
Consider a $\CC(-1)$ metric $g$ given by formula \eqref{eq:locclass}
w.r.t.\ some $A\in \Sol(g,J)$ with nonconstant eigenvalues
$\xi_1,\ldots\xi_\ell$.  Let $C_0<\cdots<C_\ell$ be the distinct eigenvalues
of the corresponding parallel hermitian endomorphism $\hat A$ on the
cone. Then $\Theta_j(t)=-4\prod_{I=0}^\ell(t-C_I)$ for $j=1,\ldots\ell$.
\end{prop}
\begin{proof}
The part of the metric $g$ in \eqref{eq:locclass} involving the $\d \xi_i$'s
corresponds to the part $\sum_{I=0}^\ell \d r_I^2$ of $g$ in
\eqref{eq:quotient}.  Using \eqref{eq:Trafo_r} and \eqref{eq:Trafo_dr}, we
obtain
\begin{align}\label{eq:dri}
4\d r_I^2=\frac{\sum_{i_1,i_2=1}^\ell\prod_{j_1\neq i_1}(C_I-\xi_{j_1})
\prod_{j_2\neq i_2}(C_I-\xi_{j_2})\d \xi_{i_1}\otimes \d \xi_{i_2}}
{\prod_{J\neq I}(C_I-C_J)\prod_{i=1}^\ell (C_I-\xi_i)}.
\end{align}
Let $4\sum_{I=0}^\ell\d r_I^2=:A_{i_1i_2}\d \xi_{i_1}\otimes \d \xi_{i_2}$.
For $i_1\neq i_2$, \eqref{eq:dri} implies that
\begin{equation*}
A_{i_1i_2}=\sum_{I=0}^\ell
\frac{\prod_{j\neq i_1,i_2}(C_I-\xi_{j})}{\prod_{J\neq I}(C_I-C_J)}.
\end{equation*}
The numerator of each term in this sum is a polynomial of degree $\ell-2$ in
$C_I$, hence, applying a Vandermonde identity (see, for instance, the appendix
of \cite{ACG1}) in the $\ell+1$ variables $C_0,\ldots C_\ell$, we see that
$A_{i_1 i_2}=0$ for $i_1\neq i_2$. For the case $i=i_1=i_2$, we obtain
\begin{align}\label{eq:Aii}
A_{ii}=\sum_{I=0}^\ell\frac{\prod_{j\neq i}(C_I-\xi_{j})}
{\prod_{J\neq I}(C_I-C_J)(C_I-\xi_{i})}.
\end{align}
The numerator of each term in this sum is a polynomial of degree $\ell-1$ in
$C_I$. Applying Vandermonde identities with respect to the $\ell+2$ variables
$C_0,\ldots C_\ell,\xi_i$, we obtain that
\begin{equation*}
A_{ii} = -\frac{\prod_{j\neq i}(\xi_i-\xi_{j})}{\prod_{I=0}^\ell (\xi_i-C_I)}.
\end{equation*}
Thus we have 
\begin{equation*}
\sum_{i=0}^\ell \d r_i^2 =-\sum_{i=1}^\ell 
\frac{\prod_{j\neq i}(\xi_i-\xi_j)}{4\prod_{I=0}^\ell(\xi_i-C_I)}\d\xi_i^2.
\end{equation*}
Comparing this with \eqref{eq:locclass}, we see that
$\Theta_i(t)=-4\prod_{I=0}^\ell(t-C_I)$ as we claimed.
\end{proof}

\subsection{$\CC(-1)$-metrics with mobility $\geq 3$}
\label{subsec:Dgeq3_rederived}

The cone construction provides a more geometric explanation why the
conditions on the eigenvalues in Theorem \ref{thm:Dgeq3} imply that the
mobility is $\geq 3$: since for a $\CC(-1)$ metric $g$ the space $\Sol(g,J)$
is isomorphic to the space of parallel hermitian endomorphisms on the cone
$(\hat M,\hat g,\hat J)$, the decomposition theorem for riemannian manifolds
\cite{ei} implies that the mobility $D(g,J)=\dim \Sol(g,J)$ is given by
\begin{align}\label{eq:mobility}
D(g,J)=f^2+i,
\end{align}
where $f$ is the complex dimension of the flat part and $i$ is the number of
irreducible (nonflat) components of $\hat g$ (see also \cite{MatRos2}). Let
$C_0\leq\cdots \leq C_n$ denote the (not necessarily distinct) eigenvalues of
a parallel hermitian endomorphism $\hat A$ on $\hat M$ and let $A$ be the
corresponding element of $\Sol(g,J)$. Lemma \ref{lem:eigvalrel} shows that
each repeated eigenvalue $C_{i-1}=C_{i}$ of $\hat A$ gives rise to a constant
eigenvalue of $A$, while each gap $C_{j-1}<C_{j}$ gives rise to a nonconstant
eigenvalue of $A$ taking values in the interval $[C_{j-1},C_j]$. This explains
the assumptions in Theorem \ref{thm:Dgeq3}: if the number of nonconstant
eigenvalues of $A$ is $\geq 2$ or the number of constant eigenvalues of $A$ is
$\geq 3$, the number of distinct eigenvalues of $\hat A$ must be $\geq
3$. Now, given a parallel hermitian endomorphism $\hat A$ on the cone with at
least three distinct eigenvalues, the decomposition theorem, together with
formula \eqref{eq:mobility} show that the mobility is $\geq 3$.

\subsection*{\bf Acknowledgements.}

We are grateful to Mike Eastwood and Katharina Neusser for pointing out a gap
(concerning the constancy of $B$) in an earlier version of the proof of
Theorem~\ref{thm:equivalence}; this result will also appear in~\cite{CEMN}
with alternative proofs of the constancy of $B$.  We are likewise indebted to
the referee for insightful comments which greatly improved the clarity and
scope of the paper. In particular, these observations suggested
Theorem~\ref{thm:orbifold} and the remarks following it.

We would also like to thank the Mathematical Sciences Institute at the
Australian National University for the opportunity to meet at a workshop in
March 2013, and the Deutsche Forschungsgemeinschaft (Research training group
1523 --- Quantum and Gravitational Fields) and FSU Jena for their financial
support.

%\nocite{*}
%\bibliographystyle{plain}

\end{document}